\documentclass{amsart}
\usepackage{graphicx}
\usepackage{amssymb}
\usepackage{amsfonts}
\usepackage{hyperref}
\usepackage{mathrsfs}
\swapnumbers
\newtheorem{theorem}{Theorem}[section]
\newtheorem{lemma}[theorem]{Lemma}
\newtheorem{corollary}[theorem]{Corollary}
\newtheorem{proposition}[theorem]{Proposition}
\theoremstyle{definition}

\newtheorem{assumption}[theorem]{Assumption}
\newtheorem{remark}[theorem]{Remark}

\numberwithin{equation}{section}
\theoremstyle{plain}

\numberwithin{equation}{section} 
\numberwithin{figure}{section} 
\theoremstyle{plain}
\theoremstyle{plain}
\theoremstyle{remark}
\newtheorem*{acknowledgement*}{Acknowledgement}
\theoremstyle{example}

\newcommand{\cG}{{\mathcal G}}

\newcommand{\te}{{\theta}}

\newcommand{\ve}{{\varepsilon}}
\newcommand{\del}{{\delta}}

\newcommand{\sig}{{\sigma}}
\newcommand{\al}{{\alpha}}

\newcommand{\ka}{{\kappa}}

\newcommand{\bbC}{{\mathbb C}}
\newcommand{\bbE}{{\mathbb E}}

\newcommand{\bbN}{{\mathbb N}}
\newcommand{\bbP}{{\mathbb P}}
\newcommand{\bbR}{{\mathbb R}}

\newcommand{\bbI}{{\mathbb I}}

\begin{document}
\title[Limit theorems]{Limit theorems for inhomogeneous random walks on $GL(d,\bbR)$}   
 \vskip 0.1cm
 \author{Yeor Hafouta}
\address{
Department of Mathematics, The University of Florida}
\email{yeor.hafouta@mail.huji.ac.il}%

\thanks{ }
\dedicatory{  }
 \date{\today}

\maketitle
\markboth{Y. Hafouta}{ } 
\renewcommand{\theequation}{\arabic{section}.\arabic{equation}}
\pagenumbering{arabic}

\begin{abstract}
We prove Berry-Esseen theorems, almost sure invariance principle rates and  large deviations for products of independent but not identically distributed invertible matrices with some average (logarithmic) projective contraction and  uniform boundedness assumptions. We also characterize the divergence of the variance of the logarithm of the norm of the product.
Our approach is based on verifying the conditions of \cite{NewBE} after reversing time. 

We then dedicate special attention to two examples. The first is  small perturbations (in a weak coupling sense) of iid and other  random matrices. More precisely we show that the average logarithmic projective contraction is  closed under appropriate perturbations.
The second example is the
the case of small perturbations of random matrices in $\text{GL}(d,\bbR)$ for which  the first singular value is larger than the second one (on average) and the distribution of the first column of the left matrix in the SVD of each matrix is sufficiently regular. For $d=2$  our conditions are related to  the distribution of the angle of the second rotation in the random singular value decomposition of the unperturbed random sequence.  What is needed is that the probability that the angle takes values in a segment of length $\epsilon$ is $O(|\ln(\epsilon)|^{-1-\alpha},\,\alpha>0$.

All of the conditions in the above examples seem to be new even for the CLT itself already in the stationary case.
Our results also holds for random contracting (in norm) matrices. Finally in the last section we discuss an extension to Markov dependent matrices.

Our results seem to be the first ones that address the CLT in beyond the $\text{SL}(2,\bbR)$ case (even for small perturbations) and the first result that provides (optimal) speed of convergence and other limit theorems. 
\end{abstract}

\section{Introduction}
The classical central limit theorem (CLT) states that if $(X_j)$ is an iid zero mean sequence of random variables in $L^2$ and $\sig^2=\bbE[X_1^2]>0$ then $(\sig\sqrt n)^{-1}S_n$ converges in distribution to the standard normal law, where $S_n=\sum_{j=1}^nX_j$. 
This is a particular case of the theorem that states that
for independent zero mean $(X_j)$ in $L^2$ then asymptotic normality is equivalent to the Lindeberg condition, where is this case the CLT means that $S_n/\|S_n\|_{L^2}$ converges to the standard normal law if $\|S_n\|_{L^2}\to\infty$. Note that here $\|S_n\|_{L^2}$ can grow arbitrarily slow.

Recall also the the Berry-Esseen theorem (see \cite{Berry, Esseen}) states that when $X_i\in L^3$ then the  CLT rate $\|S_n\|_{L^2}^{-3}\sum_{j=1}^n\bbE[|X_j|^3]$ is achieved in the uniform (Kolmogorov metric). In general the optimal rate is $O(\|S_n\|^{-1})$ and it is reached when, for instance, $X_j$ are uniformly bounded. Of course, in the iid case the more familiar rate $O(n^{-1/2})$ is achieved without further assumptions since then $\|S_n\|_{L^2}=\sigma\sqrt n$.

In this paper we are interested in non-commutative version of the above results. More precisely, we consider a sequence $(g_j)$ of random independent invertible matrices and study limit theorems for sequences of the form $S_n=\ln\|g_n\cdots g_{2}\cdot g_1x_0\|$ where $x_0$ is a fixed unit vector. 

Limit theorems for products of iid random matrices $g_j$ have been studied extensively in the past. The CLT for positive matrices was  obtained in \cite{FK61} (even for mixing stationary matrices), see also \cite{Herve}. Since then, there have been many works on limit theorems for products of iid invertible random matrices and processes with values in other groups,  including Berry-Esseen theorems and local limit theorems  (see \cite{BB,BeQu,BougLac, BRLLT05, BRLLT23, PelMat1, PelMat2, GramaBE, MatAIHP, Guiv2015, Hough2019} and references therein). Recently, also Edgeworth expansions have been obtained (see  \cite{FP, EdgeMat1, EdgeMat2}).

Very little is known about the asymptotic behavior of products of non-stationary matrices, despite that in the non-stationary commutative case has been studied extensively in recent years (see, for instance, \cite{CR07, LD1, DolgHaf PTRF 1, DolgHaf PTRF 2, DolgHaf LLT, DS,Nonlin,HafSPA, NewBE,HNTV,MPP LLT1, MPP LLT2,NSV12, NTV ETDS 18, CLT3} and references therein and  also \cite{Dob,SeVa, Pelig}).
The main problem is that all existing techniques heavily rely on tools from the theory of stationary processes (see \cite{Guiv2015,HH}). Indeed, (see \cite{Guiv2015}) a typical way is to view the product  acting on a unit vector as a Markov process on the projective space. Another approach is  based on ideas in \cite{KPZ}, that is to view the entire problem as an additive sum over a stationary Bernoulli shift (see \cite{LimMat, PelMat1}). In this paper we make a step towards closing this gap between the stationary and the non stationary cases.

To the best of our knowledge, non iid matrices  were addressed for the first time in the recent papers
\cite{Goldsh, GordKlep1, GordKlep2, GordKlep3}. In \cite{Goldsh} sufficient conditions for Markov dependent non-stationary matrices $(g_j)$ were provided that ensure that 
$$
\ln\|g_n\cdots g_{2}\cdot g_1\|
$$
grows linearly fast. In 
\cite{GordKlep1} and \cite{GordKlep2} analgous results to the law of large numbers were obtained. 
However, in general exponential growth rates are not expected. For instance, the norms of the matrices might all be smaller that some $c<1$. Yet, the logarithm of the product might exhibit a non-trivial asymptotic behavior.

As for the CLT, the only existing work is \cite{GordKlep3}. The CLT was obtained for random matrices in $\text{SL}(2,\bbR)$ under the same conditions in \cite{GordKlep1}, which ensure that the variance of $\ln\|g_n\cdots g_{2}\cdot g_1\|
$ grows linear fast. Comparing this with CLT for independent random variables $X_j$ the linear growth is a strong conclusion, since the variance can grow arbitrarily slow.

In this paper we  prove optimal CLT rates and large deviations  for invertible matrices without restrictions on the growth rate of the variance and without assumptions that lead to exponential growth of the norms.
We begin with an abstract ``clean" approach to proving limit theorems for products of independent but not identically distributed invertible random matrices. This will be done by assuming a certain type of logarithmic projective contraction. The main idea here is to show that such contraction puts us in the setup of non-stationary Bernoulli shifts which is a particular case of the setup in \cite{NewBE}. This is done in Proposition \ref{Approx Prop}. Once this proposition is proven all the main results are proven using the methods  in \cite{NewBE}.  

In Section \ref{Pert1} we will show that the logarithmic projective contraction is closed under small perturbations, which will allow us to verify it for small (random) perturbations of iid matrices, for which  the logarithmic projective contraction is known to hold. 
In Section \ref{Pert3} we discuss the general  $\text{GL}(d,\bbR)$ case and provide sufficient conditions for logarithmic contraction.
Using a somehow different approach, in Section \ref{Pert2} we will verify the contraction condition for certain classes of independent but not identically distributed $\text{SL}(2,\bbR)$-valued matrices, and thus also for their perturbations. In fact, our method works when $|\det(g_j)|=1$, and so by normalizing each $g_j$ we obtain sufficient conditions for our main results for $\text{GL}(2,\bbR)$-valued matrices. 
In Section \ref{Decay} we briefly discuss applications to contracting in norm matrices. In Section \ref{Markov} we will discuss extension to Markov dependent matrices.

Our  approach in the abstract setting is close to \cite{LimMat}. Using ideas in \cite{LimMat} and a martingale argument we show that the logarithmic contraction on average implies that when considering the standard additive cocycle decomposition.$S_n=X_1+X_2+...+X_n$, where $X_k$ depends on $g_k,...,g_1$ (see \eqref{X def}), then $X_k$ can be approximated in $L^p$ by functions of $g_k,...,g_{k-r}$ exponentially fast in $r$. For that we need to assume uniform boundedness of the matrices, which is natural when aiming at optimal CLT rates in the non-stationary setting. Note that all existing results for additive sums with optimal CLT rates in a non-stationary setting (see \cite{DolgHaf PTRF 1,DolgHaf PTRF 2, NewBE}) are obtained for uniformly bounded summands.

Our perturbative approach is based on a general simple estimate  for deterministic matrices, see Lemma \ref{ContLemma}, which compares between the (pointwise) projective Lipschitz constants of two matrices. Our approach in the general  $\text{GL}(d,\bbR)$ case involves regularity properties of the distribution of the first column $u_1$ of the left orthogonal matrix in the singular value decomposition plus a sufficiently big gap on average between the first and the second singular values. Our conditions apply to the case when uniformly in $x$ such that $\|x\|=1$ the distribution of $|<u_1,x>|$ assign mass of order $O(|\ln(\epsilon)|^{-1-\alpha}),\,\alpha>0$ to $[0,\epsilon)$.
We also provide an alternative approach in the $\text{GL}(2,\bbR)$ case which is related to the upper Hausdorff dimension of the distribution of the angle of the second rotation in the random singular value decomposition of the unperturbed random sequence (which coincides with $u_1$ above in the $2\times 2$ case).

\section{Preliminaries and main abstract results}\label{Sec2}

\subsection{The setup and standing assumptions}
Let $g_1,g_2,...$ be an independent sequence of real valued random invertible matrices of dimension $d\times d, d\geq2$ such that 
\begin{equation}\label{N g}
\sup_j\|N(g_j)\|_{L^\infty}<\infty,\,\, N(g)=\max(\|g\|,\|g^{-1}\|).    
\end{equation}
Let us fix a unit vector $x_0$ and set $S_n(x_0)=\ln\|g_n\cdots g_1x_0\|$. In this paper we prove limit theorems for the sequence of random variables $S_n(x_0)$.

Next, let $Y=P(\bbR^d)$ be the projective space and let $d(\cdot,\cdot)$ be the metric on $Y$ given by $d(\bar x,\bar y)=|x\wedge y|$, if $x$ and $y$ are unit vectors with directions $\bar x$ and $\bar y$ (that is the sine of the small angle between them).
Our standing assumption is:
\begin{assumption}[Average logarithmic contraction]\label{Ass1}
There exists $n_0\in\bbN$ and $\delta>0$ such that 
$$
\sup_j\sup_{\bar x\not=\bar y}\bbE\left[\ln\left(\frac{d(g_{j,n_0}\bar x,g_{j,n_0}\bar y}{d(\bar x,\bar y)}\right)\right]\leq -\delta.
$$
where $g_{j,n}=g_{j+n-1}\cdots g_j$.
\end{assumption}
\begin{remark}
In Section \ref{Pert1}  we show that the conditions holds for a sequence then it holds for small perturbations of the sequecne (in a certain sense of coupling). This, in particular, provides new examples for the CLT already for small perturbation of iid matrices, even iid perturbations. Other examples include  small perturbations of certain classes of independent but not identically distributed $\text{GL}(d,\bbR)$ random walks, see Sections \ref{Pert3} and \ref{Pert2}. A small comment about the case of contracting matrices appears in Section \ref{Decay}. 
\end{remark}

As usual, the starting point of studying statistical properties of $S_n(x_0)$ relies on the following additive cocycle decomposition.
Define $\sigma(g,y)=\ln\left(\frac{\|gy\|}{\|y\|}\right)$ for $y\in\bbR^d\setminus\{0\}$. Then 
\begin{equation}\label{N g1}
\sup_{y\not=0}|\sigma(g,y)|\leq |\ln(N(g))|.    
\end{equation}
Let us write 
$$
S_n(x_0)=X_1+X_2+...+X_n
$$
where 
\begin{equation}\label{X def}
X_k=\sigma(g_k,S_{k-1}(x_0))=f_k(g_k,g_{k-1},...,g_1)    
\end{equation}
where we suppress the dependence of $f_k$ on $x_0$ since it is fixed.  Then by \eqref{N g} and \eqref{N g1},
$$
\sup_k\|X_k\|_{L^\infty}<\infty.
$$
The key property of the sequence of functions $f_k$ is stated in the following result.
\begin{proposition}\label{Approx Prop}
Under \eqref{N g} and Assumption \ref{Ass1}
there exists $\rho\in(0,1)$ such that for all finite $p\geq 1$,
$$
\sup_k\sup_r\rho^{-r/p}\|X_k-\bbE[X_k|g_k,g_{k-1},...,g_{k-r}]\|_{L^p}<\infty
$$
where we extend $(g_j)$ to a two sided sequence by setting $g_k=A, k<1$ for some fixed invertible matrix $A$. 
\end{proposition}
Let us define a norm 
$$
\|f_k\|_{\infty,p,\rho^{1/p}}=\|f_k\|_{L^\infty}+\sup_r\rho^{-r/p}\|f_k-\bbE[f_k|g_k,g_{k-1},...,g_{k-r}]\|_{L^p}
$$
where we identify between $f_k$ and $X_k$. Then the proposition means that 
$$
\sup_k\|f_k\|_{\infty,p,\rho^{1/p}}<\infty
$$
for all finite $p\geq 1$.
Using the arguments in \cite{NewBE} Proposition \ref{Approx Prop} implies the following results.

\subsection{Divergence of the variance and Berry-Esseen theorems}
The first result concerns the growth rate of the variance  of $S_n(x_0)$, which from now one will be denoted by $S_n$.

In general, in order for the CLT to hold we need the individual summands to of smaller order than the variance.   In particular, we need to know when the variance is bounded. Let us begin with a characterization of this boundedness.
\begin{theorem}\label{Var them}
Under \eqref{N g} and Assumption \ref{Ass1}
the following conditions are equivalent.
\vskip0.2cm
(1) $\liminf_{n\to\infty}\text{Var}(S_n)<\infty$;
\vskip0.2cm
(2) $\sup_{n\in\bbN}\text{Var}(S_n)<\infty$;

\vskip0.2cm
(3) we can write 
$$
f_j-\bbE[f_j]=M_j+u_{j+1}\circ T_j-u_j,\,\ \mu_j-\text{a.s.}
$$
where $\sup_j\|u_j\|_{\infty,p,\rho^{1/p}}<\infty$ and
$\sup_j\|M_j\|_{\infty,p,\rho^{1/p}}\!\!<\!\!\infty$ for all finite $p\geq 1$, $u_j$ and $M_j$ have zero mean, the depend only on $g_j,...,g_1$ and $M_j(g_j,g_{j-1},...g_1),j\geq 0$ is a reverse martingale difference  with respect to the reverse filtration $\cG_j=\sigma\{g_j,g_{j-1},...,g_1\}$ and\footnote{Note that by the martingale converges theorem we get that the sum $\sum_{j=0}^\infty M_j(g_j,g_{j-1},...,g_1)$ converges almost surely and in $L^s$.} 
$$
\sum_{j\geq 0}\text{Var}(M_j(g_j,g_{j+1},...,g_1))\!\!<\!\!\infty.
$$
\vskip0.2cm
(4) there exist measurable functions $H_j$ such that 
$$
f_j(g_j,g_{j-1},...,g_1)=H_{j+1}(g_{j+1},g_j,...,g_1)-H_j(g_j,g_{j-1},...,g_1),\,\text{ a.s.}
$$
\end{theorem}
We also refer to Proposition \ref{MomAss} for moment estimates of $S_n$.

Next, denote $\sig_n=\sqrt{\text{Var}(S_n)}$.
Recall that $S_n$ obeys the central limit theorem (CLT) if for every real $t$, 
$$
F_n(t):=\bbP(S_n-\bbE[S_n]\leq\sig_n t)\to\frac{1}{\sqrt{2\pi}}\int_{-\infty}^te^{-x^2/2}dx:=\Phi(t)
$$
as $n\to\infty$. Our next results are optimal convergence rates in the CLT, but we stress that the CLT by itself seems to be a new results in our setup.

\begin{theorem}\label{BE}
Under \eqref{N g} and Assumption \ref{Ass1} we have:
\vskip0.2cm
(i)  for all finite $s\geq 0$ there is a constant $C_s$ such that 
$$
\sup_{t\in\bbR}(1+|t|^s)\left|F_n(t)-\Phi(t)\right|\leq C_s\sig_n^{-1}.
$$ 
\vskip0.1cm
(ii) for all $q>0$ we have 
$
\left\|F_n-\Phi\right\|_{L^q(dx)}=O(\sig_n^{-1}).
$
\vskip0.1cm
(iii)  for all finite $s\geq 1$  there is a constant $C_s$ such that
for every absolutely continuous function $h:\bbR\to\bbR$  such that 
$H_s(h):=\int \frac{|h'(x)|}{1+|x|^s}dx<\infty$ we have 
$$
\left|\bbE[h((S_n-\bbE[S_n])/\sig_n)]-\int h d\Phi\right|\leq C_s H_s(h)\sig_n^{-1}.
$$
\end{theorem}
One example of functions in (iii) above are $h(x)=x^{a}, a<s$. Then Theorem \ref{BE} (iii) provides estimates form the moments of $S_n-\bbE[S_n]$ by means of the variance of $S_n$ and the standard normal moments.

Next, recall that the $p$-th Wasserstien distance between two probability measures $\mu,\nu$ on $\bbR$ with finite absolute moments of order $b$ is given by 
$$
W_p(\mu,\nu)=\inf_{(X,Y)\in\mathcal C(\mu,\nu)}\|X-Y\|_{L^b}
$$
where $\mathcal C(\mu,\nu)$ is the class of all  pairs of random variables $(X,Y)$ on $\bbR^2$ such that $X$ is distributed according to $\mu$, and $Y$ is distributed according to $\nu$. 

\begin{theorem}\label{ThWass}
Under \eqref{N g} and Assumption \ref{Ass1}, for every finite $p\geq 1$ we have 
$$
W_p(dF_n, d\Phi)=O(\sig_n^{-1})
$$
where $dG$ is the measure induced by a distribution function $G$.
\end{theorem}

\subsection{Almost sure invariance principle rates}
\begin{theorem}\label{ASIP}
For every $\varepsilon>0$ there is a coupling of the sequence $(g_j)$ with a Brownian motion $W(t)$ such that 
\vskip0.1cm

(1)$$|S_n-\bbE[S_n]-W(\sigma_n^2)|=O(\sigma_n^{1/2+\varepsilon})$$ 
almost surely, and
\vskip0.1cm
(2)$$\|S_n-\bbE[S_n]-W(\sigma_n^2)\|_{L^2}=O(\sigma_n^{1/2+\varepsilon}).$$
\end{theorem}

\subsection{Large deviations}
Our next result is an exponential concentration inequality.
\begin{theorem}\label{ConcentrationIneq}
 There exist constants $c,C>0$ such that for all $t>0$ and all $n$,
 $$
\bbP(|S_n|\geq tn+C)\leq 2e^{-ct^2n}.
 $$
\end{theorem}

\subsection{Moderate deviations principle under linear growth of the variance}

We can prove the following moderate deviations principle  with optimal scale.
\begin{theorem}\label{MDP}
Under \eqref{N g} and Assumption \ref{Ass1} the following holds. Suppose that $\liminf\frac{\sigma_n}{\sqrt n}>0$.
 Let $(a_n)$ be a sequence such that $a_n\to\infty$ such that $\lim_{n\to\infty}\frac{a_n}{\sqrt n}=\infty$
 but $a_n=o(n)$. Denote $s_n=a_n^2/n$.
Then for every Borel measurable set $\Gamma\subset\bbR$
$$
-\frac12\inf_{x\in \Gamma^o}x^2\leq \liminf_{n\to\infty}\frac{1}{s_n}\ln\bbP((S_{n}f/a_n)\in\Gamma)\leq \limsup_{n\to\infty}\frac{1}{s_n}\ln\bbP((S_{n}f/a_n)\in\Gamma)\leq -\frac12\inf_{x\in \overline{\Gamma}}x^2
$$
where $\Gamma^o$ is the interior of $\Gamma$ and $\overline{\Gamma}$ is it's closure. 
\end{theorem}
In Section \ref{Pert1} we will discuss for which types of perturbations of iid matrices we get the linear growth of the variance. Let us also refer to \cite{KiferRandRand} for certain random products in random ergodic environment for which we get the linear growth.

\section{Proof of the main abstract results}
\subsection{Proof of Proposition \ref{Approx Prop}}
\begin{proposition}\label{Exp Prop}
There exist $\ell,C>0$ and $\gamma\in(0,1)$ such that for every $j$ and $k$ we have
$$
\sup_{\bar x\not=\bar y}\bbP\left(\ln(d(g_{j,k}\bar x, \bar g_{j,k}\bar y)\geq -\ell k\right)\leq C\gamma^{k}.
$$
\end{proposition}
\begin{proof}
The proof is a modification of the proof of \cite[Lemma 6]{LimMat}. The main difference here is that we will get uniformly bounded martingales and so we can apply the Azuma inequality instead of complete convergence for martingales. See the end of the proof.    Let us define 
$$
F_j(\bar x,\bar y)=\bbE\left[\ln\left(\frac{d(g_j\bar x,g_j\bar y)}{d(\bar x,\bar y)}\right)\right]
$$
and 
$$
\sig_j(g, (\bar x,\bar y))=\ln\left(\frac{d(g\bar x,g\bar y)}{d(\bar x,\bar y)}\right)-F_j(\bar x,\bar y).
$$
Then 
$$
\ln\left(\frac{d(g_{j,k}\bar x, \bar g_{j,k}\bar y}{d(\bar x,\bar y)}\right)=M_{j,k}+R_{j,k}
$$
where 
$$
R_{j,k}=R_{j,k}(\bar x,\bar y)=\sum_{m=1}^{k}F_{j+m}(g_{j,m}\bar x,g_{j,m}\bar y)
$$
and 
$$
M_{j,k}=\sum_{m=1}^{k}\sigma_{j+m}(g_{j+m},(g_{j,m}\bar x,g_{j,m}\bar y)).
$$
Note that for every fixed $j$ we have that $\sigma_{j+m}(g_{j+m},(g_{j,m}\bar x,g_{j,m}\bar y))$ is uniformly (in $j,m$ and $\bar x$ and $\bar y$) bounded martingale difference (as $N(g_s), s\geq 1$ are uniformly bounded random variables).
Notice that by Assumption \ref{Ass1} for all $j$ and unit vectors $\bar x,\bar y$ we have 
$$
\bbE[R_{j,n_0}(\bar x,\bar y)]\leq -\delta.
$$
Using that, repeating the arguments in the proof of  \cite[Lemma 6]{LimMat} we see that there exist constants $C,\alpha>0$ and $\eta\in(0,1)$ such that 
$$
\sup_j\bbP(R_{j,k}(\bar x,\bar y)\geq -\alpha k)\leq C\eta^k.
$$
Finally, by applying the Azuma inequality and using the Chernoff bounding method we conclude that there is a constant $c>0$ such that for every $j,k$ and $\varepsilon>0$ we have 
$$
\bbP(M_{j,k}\geq \varepsilon k)\leq e^{-c\varepsilon^2 k}.
$$
By taking $\varepsilon$ small enough we get the desired result.
\end{proof}

\begin{proof}[Proof of Proposition \ref{Approx Prop}]
First, since $X_k$ are uniformly bounded it is enough to prove the proposition in the case $p=1$.
Next, by \cite[Lemma 12.2]{BeQu} for every two unit vectors $x,y$ we have 
\begin{equation}\label{Lip}
\|\sigma(g,y)-\sigma(g,x)\|\leq CN(g)d(x,y).    
\end{equation}
Denote by $X_{k,x}$ the value of $X_k$ when $x_0=x$. Let us take two unit vectors $x$ and $y$ and set $B=\{\ln(d(S_{k-1}(x),S_{k-1}(y))\geq -\ell k)$. Then
$$
|X_{k,x}-X_{k,y}|=|\sigma(g_k,S_{k-1}(x))-\sigma(g_k,S_{k-1}(y))|\bbI_{B}+|\sigma(g_k,S_{k-1}(x))-\sigma(g_k,S_{k-1}(y))|\bbI_{B^c}.
$$
Now, as $N(g_k)$ are uniformly bounded, using also \eqref{Lip} , we conclude that 
$$
\|X_{k,x}-X_{k,y}\|_{L^1}\leq C_1\bbP(B)+C_1e^{-\ell k}.
$$
Now by Proposition \ref{Exp Prop} we have $\bbP(B)\leq C\gamma^k$. Thus there exists $\delta\in(0,1)$ and a constant $C_0>0$ such that 
$$
\|X_{k,x}-X_{k,y}\|_{L^1}\leq C_0\delta^k.
$$
Fixing some $j<k$ and starting the multiplication from $g_j$ instead of $g_1$ we see that
$$
\|\sigma(g_{k}, g_{j,k-1-j}x)-\sigma(g_{k}, g_{j,k-1-j}y)\|_{L^1}\leq C\delta^{k-j}.
$$
Therefore, if we denote by $\mu_j$ the law of $g_j$ then
$$
\|f_{k}(g_k,g_{k-1},...,g_1)-\sigma(g_{k}, g_{j,k-1-j}x_0)\|_{L^1}
$$
$$
=\int\bbE\left[|f_{k}(g_k,g_{k-1},...,g_j,h_{j-1},...,h_1)-\sigma(g_{k}, g_{j,k-1-j}x_0)|\right]d\mu_1(h_1)\ldots d\mu_{j-1}(h_{j-1})\leq  C\delta^{k-j}.
$$
Hence, by the minimization property of condition expectations,
$$
\|f_{k}(g_k,g_{k-1},...,g_1)-\bbE[f_{k}(g_k,g_{k-1},...,g_1)|g_k,...,g_{j}]\|_{L^1}\leq C\delta^{k-j}
$$
and the proof of the proposition is complete.
\end{proof}

\subsection{Reversing time}
Let us take a fixed invertible matrix $A$ which is logarithmically contracting and set $g_k=A$ for $k<1$. Define $Y_k=g_{-k}$ which is also a sequence of independent invertible matrices. So we get 
$$
f_k(g_{k},g_{k-1},...,g_1)=f_k(Y_{-k},Y_{-k+1},....).
$$
For all complex $z$ define an operator $L_{j,z}$ which maps a function $g$ on $X=(\text{GL}(d,\bbR))^{\bbN}$ to a function $L_{j,z}g$  on $X$ given by   
$$
L_{j,z}g(x)=\bbE[g(Y_{-j},Y_{-j+1},....)e^{zf_j(Y_{-j},Y_{-j+1},....)}|(Y_{-j+1},Y_{-j+2},...)=x].
$$
Let $\rho$ be the number from Proposition \ref{Approx Prop} and let us define a norm
$$
\|g\|_{j,\rho}=\|g(Y_{-j},Y_{-j+1},....)\|_{L^\infty}+\sup_r\rho^{-r}\|f_k-\bbE[f_k|g_k,g_{k-1},...,g_{k-r}]\|_{L^1}.
$$
The following theorem is proven exactly like in \cite{NewBE}. Denote by $B$ the space of function $g$ with $\|g\|_{\infty,1,\rho}<\infty$. Denote by $\kappa_j$ the law of $(Y_{-j},Y_{-j+1},...)$.
\begin{theorem}\label{RPF}
There exists $0<\delta_0<1$ such that for every $z\in\bbC$ with $|z|\leq\delta_0$ there are $\lambda_j(z)\in\bbC\setminus\{0\}$, $h_j^{(z)}\in B$ and $\ka_j^{(z)}\in B^*$ such that $\mu_j^{(Z)}(\textbf{1})=\mu_j^{(Z)}(h_j^{(Z)})=1$, 
 $\lambda_j(0)=1$, $h_j^{(0)}=\textbf{1}$, $\ka_j^{(0)}=\ka_j$
and
\begin{equation}\label{UP}
  L_{j,z} h_j^{(z)}=\lambda_j(z)h_{j+1}^{(z)},\, (L_{j,z})^*\ka_{j+1}^{(z)}=\lambda_j(t)\ka_{j}^{(z)}.   
\end{equation}
Moreover, $t\to\lambda_j(z)$, $t\to h_j^{(z)}$ and $t\to\mu_j^{(z)}$ are analytic functions of $z$ with uniformly (over $z$ and $j$) bounded norms.. 
Finally,  there are $C_1>0, \delta_1\in(0,1)$ such that for every $g\in B$ and all $n$,
\begin{equation}\label{Exp}
 \left\|L_{j,z}^{n}g-\lambda_{j,n}(z)\ka_j^{(z)}(g)h_{j+n}^{(t)}\right\|_{\infty,1,\rho}\leq C_1\|g\|_{\infty,1,\rho}\delta_1^n   
\end{equation}
where $\lambda_{j,n}(z)=\prod_{k=j}^{j+n-1}\lambda_k(z)$ and 
$L_{j,z}^{n}=L_{j+n-1,z}\circ\cdots\circ L_{j,z}$.
\end{theorem}

Using this theorem we are able to prove the following results similarly to \cite{NewBE}, and we leave the notational modifications to the reader.

\begin{proposition}\label{MartProp}
We can write 
$$
f_j-\bbE[f_j]=M_j+u_{j+1}\circ T_j-u_j,\,\ \mu_j-\text{a.s.}
$$
where $\sup_j\|u_j\|_{\infty,p,\rho^{1/p}}<\infty$ and
$\sup_j\|M_j\|_{\infty,p,\rho^{1/p}}\!\!<\!\!\infty$ for all finite $p\geq 1$, $u_j$ and $M_j$ have zero mean, the depend only on $g_j,...,g_1$ and $M_j(g_j,g_{j-1},...g_1),j\geq 0$ is a reverse martingale difference  with respect to the reverse filtration $\cG_j=\sigma\{g_j,g_{j-1},...,g_1\}$. 
\end{proposition}
This result together with the main results in \cite{BDH} complete the proof of Theorem \ref{Var them}.

\begin{proposition}\label{MomAss}
For every $p\geq 2$ there exists a constant $C_p>0$ such that for every $j$ and $n$ we have 
$$
\|S_{j,n}-\bbE[S_{j,n}]\|_{L^p}\leq C_p(1+\|S_{j,n}-\bbE[S_{j,n}]\|_{L^2})
$$
where $S_{j,n}=X_{j}+X_{j+1}+...+X_{j+n-1}$.
\end{proposition}

and 
\begin{proposition}\label{LogGrowth}
There are constants
 $\del_k>0$ and $C_k>0$ such that for all $s\in\bbN$ we have
\begin{equation}\label{Add}
\sup_{t\in[-\del_k,\del_k]}|\Lambda^{(s)}_{0,n}(t)|\leq C_3\sig_n^2.   
\end{equation}
where $\Lambda_{0,n}(t)=\ln(\bbE[e^{it S_n}])$.
\end{proposition}

Using \ref{Add} all the Berry-Esseen theorems follow from the main result in \cite[Theorem 4 and Corollary 5]{NonUBE} and \cite[Theorem 9]{NonUEdge}.  Theorem \ref{MDP}  follows directly from Theorem \ref{RPF} like in \cite{NewBE}, while the proof of Theorem \ref{ConcentrationIneq} follows by the applying the Azuma-Hoeffding inequality and using the Chernoff bounding scheme.

The proof of Theorem \ref{ASIP} follow using similar argument to \cite[Theorem ]{DolgHaf PTRF 2 ARXIV}.

\section{Applications}
\subsection{Small perturbations of random contracting matrices}\label{Pert1}
In this section we will show that Assumption \ref{Ass1} is closed under certain type of perturbations. We thus get that all our results hold true for small perturbations of random iid strongly irreducible and proximal matrices\footnote{Namely, the group generated by the support of the common distribution on $\text{GL}(d,\bbR)$ is strongly irreducible and contains a proximal element}  (see Remark \ref{RemIID}). In the next sections we will study in detail small perturbations of certain classes of non-stationary independent $\text{GL}(D,\bbR)$-valued matrices. 

We first need the following relatively elementary result.
\begin{lemma}\label{ContLemma}
For two invertible matrices $A$ and $B$ and distinct directions $\bar x,\bar y\in P(\bbR^d)$ we have
$$
\frac{d(A\bar x,A\bar y)}{d(\bar x,\bar y)}\leq c(A,B)+\frac{d(B\bar x,B\bar y)}{d(\bar x,\bar y)} 
$$
where 
$$
c(A,B)=(\|A\|+\|B\|)\|A-B\|\left(\frac{1}{\sigma_{min}^2(A)}+\frac{\|B\|^2}{\sigma_{min}^2(A)\sigma_{min}^2(B)}\right)
$$
and $\sigma_{min}(\cdot)$ is the smallest singular value (the minimum on the unit circle).
\end{lemma}
\begin{proof}
Let $x,y$ be two point on the unit circle with directions $\bar x$ and $\bar y$.
Denote 
$p=d(\bar x,\bar y)=\|x\wedge y\|$, $a=\|A x\wedge A y\|$, $\alpha=\|Ax\|\|Ay\|$, $b=\|B x\wedge B y\|$, $\beta=\|Bx\|\|By\|$.
We want to show that
$$
\frac{a}{\alpha}-\frac{b}{\beta}\leq pc(A,B).
$$
Write 
$$
\frac{a}{\alpha}-\frac{b}{\beta}=\frac{a-b}{\alpha}+b\left(\frac{1}{\alpha}-\frac{1}{\beta}\right).
$$
We have $1/\alpha\leq 1/\sigma_{min}^2(A)$. Now, we also have 
$$
|a-b|\leq \|Ax\wedge Ay-Bx\wedge By\|\leq\|(A-B)x\wedge Ay\|+\|Bx\wedge((A-B)y)\|. 
$$
The operator norm on $\wedge^2$ gives 
$$
\|((A-B)x)\wedge Ay\|\leq \|A-B\|\|A\|\|x\wedge y\|=\|A-B\|\|A\|d
$$
and similarly
$$
\|Bx\wedge((A-B)y)\|\leq \|A-B\|\|B\|d.
$$
Thus,
$$
\left|\frac{a-b}{\alpha}\right|\leq\|A-B\|(\|A\|+\|B\|)d (\sig_{min}(A))^{-2}. 
$$
Next,
$$
\left|\frac{1}{\alpha}-\frac{1}{\beta}\right|\leq 
|\alpha-\beta|(\sig_{min}(A)\sig_{min}(B))^{-2}.
$$
It is also clear that 
$$
|\alpha-\beta|\leq (\|A\|+\|B\|)\|A-B\|.
$$
As $b=\|Bx\wedge By\|\leq \|\wedge^2B\|p\leq\|B\|^2p$ we see that 
$$
b\left|\frac{1}{\alpha}-\frac{1}{\beta}\right|\leq (\sig_{min}(A)\sig_{min}(B))^{-2}\|B\|^2p (\|A\|+\|B\|)\|A-B\|.
$$
\end{proof}
\begin{remark}
 Note that for every invertible matrix $A$ and a unit vector $x$ we have $\|Ax\|\geq \|A^{-1}\|^{-1}$. Therefore we can replace $(\sig_{min}(A))^{-1}$ and $(\sig_{min}(B))^{-1}$ by $\|A^{-1}\|$ and $\|B^{-1}\|$, respectively.    Hence,
 $$
\frac{d(A\bar x,B\bar y)}{d(\bar x,\bar y)}\leq \tilde c(A,B)+\frac{d(B\bar x,B\bar y)}{d(\bar x,\bar y)} 
$$
where 
$$
\tilde c(A,B)=(\|A\|+\|B\|)\|A-B\|\left(\|A^{-1}\|^2+\|B\|^2\|A^{-1}\|^2\|B^{-1}\|^2\right).
$$
\end{remark}
\begin{corollary}\label{Corr}
Let $(g_j)$ and $(h_j)$ be two independent sequences of random invertible matrices such that Assumption \ref{Ass1} holds for the sequence $(h_j)$. If we can couple $(h_j)$ and $(g_j)$ such that  $\te:=\sup_j\bbE[\tilde c(g_{j,n_0},h_{j,n_0})]<\infty$. Then there exists a constant $\varepsilon_0=\varepsilon_0(\delta,n_0)$ such that  Assumption \ref{Ass1} holds for the sequence $(g_j)$ with some $\delta_1>0$ instead of $\delta$ and the same $n_0$ if $\te<\varepsilon_0$.

In particular, suppose that
 $C_1=\sup_j\|1+N(h_j)\|_{L^{8n_0}}<\infty$ and $C_2=\sup_j\|1+N(g_j)\|_{L^{8n_0}}<\infty$. Assume also
that we can couple $(g_j)$ and $(h_j)$ such that $\sup_j\|h_j-g_j\|_{L^{8n_0}}\leq \varepsilon$. Then there exists a constant $\varepsilon_0=\varepsilon_0(\delta,n_0,C_1,C_2)$ such that  Assumption \ref{Ass1} holds for the sequence $(g_j)$ with some $\delta_1>0$ instead of $\delta$ and the same $n_0$ if $\varepsilon<\varepsilon_0$.
\end{corollary}
\begin{proof}
Fix some $j$ and two distinct directions $\bar x$ and $\bar y$. Set
$$
F=\tilde c(g_{j,n_0},h_{j,n_0}),\, G=\frac{d(g_{j,n_0}\bar x,g_{j,n_0}\bar y)}{d(\bar x,\bar y)},\,  H=\frac{d(h_{j,n_0}\bar x,h_{j,n_0}\bar y)}{d(\bar x,\bar y)}.
$$
Then by the previous lemma
$$
G\leq F+H.
$$
Fix some $a>0$ and $\gamma>0$ such that $1-(\gamma-a)<1$ and let $\Gamma=\{F\leq a\}$ and $\Delta=\{H\leq 1-\gamma\}$. Then 
$$
\ln(G)\bbI_{\Gamma}\bbI_{\Delta}\leq \ln(1-(\gamma-a)).
$$
On the other hand, we have 
$$
G\leq\|g_{j,n_0}\|^2\leq  \prod_{k=j}^{j+n_0-1}\|g_k\|^2.
$$
Therefore, 
$$
\ln(G)\bbI_{\Gamma^c}\leq 2\sum_{k=j}^{j+n_0-1}\ln(\|g_k\|)\bbI_{\Gamma^c}.
$$
Notice now that by the Markov inequality,
$$
\bbP(\Gamma^c)=\bbP(F>a)\leq \bbE[F]/a
$$
and so by the Cauchy-Schwarz inequality,
$$
\bbE[\ln(G)\bbI_{\Gamma^c}]\leq 2n_0\sup_j\|\ln(\|g_j\|)\|_{L^2}a^{-1/2}(\bbE[F])^{1/2}.
$$
Next, notice that $\ln(G)\leq \ln(F+H)\leq \ln (H)+\frac{F}{H}$. Thus,
$$
\bbE[\ln(G)\bbI_{\Delta^c}]\leq \bbE[\ln(H)]+(1-\gamma)^{-1}\bbE[F].
$$
Now, notice that as $\varepsilon\to 0$ we have that $\bbE[F]\to 0$. Thus the result follows for $\varepsilon$ small enough (which allows us to take $a$ small but not too small to ensure that $\bbE[F]/a$ is small).
\end{proof}
\begin{remark}\label{RemIID}
In applications when $h_j$ is an strongly irreducible and proximal iid sequence then (see \cite{Guiv2015}) we get that for some $\al\in(0,1]$, 
$$
\rho(n_0)=\sup_{\bar x\not=\bar y}\bbE\left[\frac{d^\alpha(h_{j,n_0}\bar x,h_{j,n_0}\bar y)}{d^\alpha(\bar x,\bar y)}\right]<1
$$
where we note that the above expectation does not depend on $j=0$ due to stationarity.
Using the Jensen inequality with the logarithmic function this is much stronger than Assumption \ref{Ass1}.  
In that case we see directly that 
$$
\bbE\left[\frac{d^\alpha(g_{j,n_0}\bar x,g_{j,n_0}\bar y)}{d^\alpha(\bar x,\bar y)}\right]\leq \bbE\left[\tilde c^\alpha(g_{j,n_0},h_{j,n_0})\right]+\bbE\left[\frac{d^\alpha(h_{j,n_0}\bar x,h_{j,n_0}\bar y)}{d^\alpha(\bar x,\bar y)}\right]. 
$$
Assume we can couple $(h_j)$ and $(g_j)$ such that  $\te:=\sup_j\bbE[\tilde c^\alpha(g_{j,n_0},h_{j,n_0})]<\infty$. Then there exists a constant $\varepsilon_0=\varepsilon_0(\delta,n_0)$ such that  Assumption \ref{Ass1} holds for the sequence $(g_j)$ with some $\delta_1>0$ instead of $\delta$ and the same $n_0$ if $\te<\varepsilon_0$.

In particular, suppose that
 $C_1=\sup_j\|1+N(h_j)\|_{L^{8\al n_0}}<\infty$ and $C_2=\sup_j\|1+N(g_j)\|_{L^{8\al n_0}}<\infty$. Assume also
that we can couple $(g_j)$ and $(h_j)$ such that $\sup_j\|h_j-g_j\|_{L^{8\al n_0}}\leq \varepsilon$. Then there exists a constant $\varepsilon_0=\varepsilon_0(\delta,n_0,C_1,C_2)$ such that if $\varepsilon<\varepsilon_0$ then
$$
\bbE\left[\frac{d^\alpha(g_{j,n_0}\bar x,g_{j,n_0}\bar y)}{d^\alpha(\bar x,\bar y)}\right]<1
$$
which again is much stronger than Assumption \ref{Ass1}.
\end{remark}

\begin{remark}[Linear growth of the variance]
 Let $(h_j)$ be any independent sequence of random invertible matrices such that 
$$
\liminf_{n\to\infty}\frac{1}{n}\text{Var}(\|h_n\cdots h_1 x_0\|)>0.
$$   
Then in \cite[Section 5.3.2]{NonUBE} we showed that if $N(h_j)$ is uniformly bounded and $\sup_j\|\mu_j-\nu_j\|_{TV}$ is small enough, where $\mu_j$ is the law of $g_j$ and $\nu_j$ is the law of $h_j$, then 
$$
\liminf_{n\to\infty}\frac{1}{n}\text{Var}(\|g_n\cdots g_1 x_0\|)>0.
$$  
Thus, the linear growth assumption in Theorem \ref{MDP} is satisfied. 
\end{remark}

\subsection{Small perturbations of inhomogeneous random walks on $GL(d,\bbR)$}\label{Pert3}
Let $G$ be a random matrix with values in $\text{GL}(d,\bbR)$. Let us consider a random SVD for $G$,
$$
G=U\text{Diag}(\sigma_1,\sigma_2,...,\sigma_d)V
$$
and let $u_i$ be the random columns of $U$. Notice that for two unit vectors $x,y$ we have
$$
\frac{\|Gx\wedge Gy\|^2}{\|x\wedge y\|^2\|Gx\|^2\|Gy\|^2}=\frac{\sum_{1\leq i<j\leq d}\sigma_i^2\sigma_j^2 A_{i,j,U}(x,y)}{\sum_{1\leq i,j\leq d}\sigma_j^2\sigma_i^2a_{i,j,U}(x,y)}
$$
where 
$$
A_{i,j,U}(x,y)=\frac{\left(<u_i,x><u_j,y>-<u_i,y><u_j,x>\right)^2}{\sum_{1\leq k<\ell\leq d}\left(<u_k,x><u_\ell,y>-<u_k,y><u_\ell,x>\right)^2}
$$
and 
$$
a_{i,j,U}(x,y)=\left(<u_i,x>\right)^2\left(<u_j,y>\right)^2.
$$
Note that both $(A_{i,j,U}(x,y))_{i<j}$ and $(a_{i,j,U}(x,y))_{i,j}$ are probability vectors. Therefore,
$$
\frac{\|Gx\wedge Gy\|^2}{\|x\wedge y\|^2\|Gx\|^2\|Gy\|^2}\leq \frac{\sigma_1^2\sigma_2^2 }{\sigma_1^4a_{1,1,U}(x,y)}.
$$

By combining the above  we thus get the following results.
\begin{proposition}\label{Prrrp}
Suppose that there is a constant $\delta>0$ such that for every unit vector $x$ we have
\begin{equation}\label{Condd}
2\bbE\left[\left|\ln(|<u_1,x>|)\right|\right]\leq \bbE[\ln(\sig_1/\sig_2)]-\delta.    
\end{equation}
Then 
$$
\sup_{x\not=y}\bbE\left[\ln\left(\frac{\|Gx\wedge Gy\|}{\|x\wedge y\|\|Gx\|\|Gy\|}\right)\right]\leq -\delta.
$$
where $x,y$ above are unit vectors.
\end{proposition}

\begin{corollary}
Let $h_j$ be a sequence of independent random matrices in $\text{GL}(d,\bbR)$ satisfying the conditions of Proposition \ref{Prrrp} with the same constant $\delta$. Let $g_j$ be another sequence of  independent random matrices  in $\text{GL}(d,\bbR)$ satisfying \eqref{N g}. Suppose that we can couple both sequence such that either $h_j$ satisfies \eqref{N g} and $\sup_j\|g_j-h_j\|_{L^1}$ is small enough or $\sup_j\|N(h_j)\|_{L^8}<\infty$ and $\sup_j\|g_j-h_j\|_{L^8}$ is small enough. Then all the results in Section \ref{Sec2} hold for the sequence $g_{n}\cdots g_1$.
\end{corollary}

To verify \eqref{Condd} we need the following result.
\begin{lemma}
Let $X$ be a random variable with values in $[0,1]$ and suppose that there are $C,\alpha>0$ such that for all $\epsilon\in(0,1)$, 
$$
\bbP(X\leq \epsilon)\leq C|\ln(\epsilon)|^{-1-\alpha}.
$$
Then with $R(C,\alpha)=1+C/\alpha$ we have 
$$
\bbE[|\ln(X)|]\leq R(C,\alpha).
$$
\end{lemma}
\begin{proof}
 $$
\bbE[|\ln(X)|]=\int_{0}^\infty\bbP(|\ln(X)|\geq t)dt=\int_{0}^{\infty}\bbP(\ln(X)\leq -t)dt
$$
$$
=\int_{0}^\infty\bbP(X\leq e^{-t})dt\leq 1+\int_{1}^\infty\bbP(X\leq e^{-t})dt\leq R(C,\alpha).
$$   
\end{proof}
By applying the lemma with $X=|<u,x>|$ we get the following result.
\begin{corollary}\label{SupCor}
Suppose that there are constants $C,\alpha>0$ such that 
$$
\sup_{\|x\|=1}\bbP(|<u_1,x>|\leq \epsilon)\leq C(|\ln(\epsilon)|)^{-1-\alpha}
$$
for every $\epsilon\in(0,1)$. Assume also that there is a constant $\delta>0$ such that 
$$
\bbE[\ln(\sig_1/\sig_2)]\geq R(C,\alpha)+\delta=1+C/\alpha+\delta.
$$
Then 
$$
\sup_{x\not=y}\bbE\left[\ln\left(\frac{\|Gx\wedge Gy\|}{\|x\wedge y\|\|Gx\|\|Gy\|}\right)\right]\leq -\delta.
$$
\end{corollary}

\begin{corollary}
Let $h_j$ be a sequence of independent random matrices in $\text{GL}(d,\bbR)$ satisfying the conditions of Corollary \ref{SupCor} with the same constants $c,\alpha,\delta$. Let $g_j$ be another sequence of  independent random matrices  in $\text{GL}(d,\bbR)$ satisfying \eqref{N g}. Suppose that we can couple both sequences such that either $h_j$ satisfies \eqref{N g} and $\sup_j\|g_j-h_j\|_{L^1}$ is small enough or $\sup_j\|N(h_j)\|_{L^8}<\infty$ and $\sup_j\|g_j-h_j\|_{L^8}$ is small enough. Then all the results in Section \ref{Sec2} hold for the sequence $g_{n}\cdots g_1$.
\end{corollary}

\subsection{Small perturbations of inhomogeneous random walks on $GL(2,\bbR)$-another approach}\label{Pert2}
In this section we will develop a different approach in the case $d=2$.
We will focus here on the case when $|\det(g_j)|=1$ and at the end of the section we will comment about the general $2\times 2$ case.
Let $(h_j)$ be a sequence of independent matrices such that $|\det(h_j)|=1$ almost surely.


\begin{proposition}\label{Prp}
$(h_j)$ obeys Assumption \ref{Ass1} if there exists  $\varepsilon>0$ such that 
\begin{equation}\label{prop}
 \sup_j\sup_{\|x\|=1}\bbE[\|h_{j,n_0}x\|^{-2\varepsilon}]<1.    
\end{equation}
\end{proposition}

\begin{proof}
 Let $X$ be a positive random variable. Then by Jensen inequality for all $\varepsilon>0$ we have
$$
\varepsilon\bbE[\log(X)]=\bbE[\log(X^\varepsilon)]\leq \log(\bbE[X^\varepsilon]).
$$
Taking $X=\frac{d(h_{j,n_0}\bar x,h_{j,n_0}\bar y)}{d(\bar x,\bar y)}$, $\bar x\not=\bar y$ we get that 
\begin{equation}\label{Upp}
\varepsilon\bbE\left[\log\left(\frac{d(h_{j,n_0}\bar x,h_{j,n_0}\bar y)}{d(\bar x,\bar y)}\right)\right]\leq \log\left(\bbE\left[\frac{d^\varepsilon(h_{j,n_0}\bar x,h_{j,n_0}\bar y)}{d^\varepsilon(\bar x,\bar y)}\right]\right).
\end{equation}
Thus, it is enough to prove that there exists $\varepsilon>0$ such that 
$$
\sup_{j}\sup_{\bar x\not=\bar y}\bbE\left[\frac{d^\varepsilon(h_{j,n_0}\bar x,h_{j,n_0}\bar y)}{d^\varepsilon(\bar x,\bar y)}\right]<1.
$$
Now, using that  $d(h\bar x,h\bar y)=d(\bar x,\bar y)\|hx\|^{-1}\|hy\|^{-1}$ for every $h\in\text{GL}(2,\bbR)$ with $|\det(h)|=1$  and unit vectors $x,y$ (see \cite[(2), page 18]{BougLac}) we get from the Cauchy-Schwartz inequality that
$$
\bbE\left[\frac{d^\varepsilon(h_{j,n_0}\bar x,h_{j,n_0}\bar y)}{d^\varepsilon(\bar x,\bar y)}\right]
\leq \sup_{\|x\|=1}\bbE[\|h_{j,n_0}x\|^{-2\varepsilon}].
$$
\end{proof}

Next, let $h\in\text{GL}(2,\bbR)$ with $|\det(h)|=1$. Recall that the largest singular value of a matrix is its operator norm and that the smallest singular value is the minimum on the unit circle. Consider a singular value decomposition (SVD) of a matrix $h$ with determinant $\pm1$,
$$
h=U\begin{pmatrix}
a& 0 \\
0 & a^{-1}
\end{pmatrix}	
V 
$$
with $U$ and $V$ unitary and $a=\|h\|$. Then 
$$
\inf_{\|x\|=1}\|hx\|=\|h\|^{-1}.
$$

Next consider a random matrix $G$ such that $|\det(G)|=1$ a.s. and consider a random SVD 
$$
G=U\begin{pmatrix}
\|G\|& 0 \\
0 & \|G\|^{-1}
\end{pmatrix}	
V 
$$
with $U$ and $V$ random unitary matrices. Then, 
$$
G=\|G\|u_1\otimes v_1+\|G\|^{-1}u_2\otimes v_2.
$$
Here $u_i$ and $v_i$ are the random columns and rows of the random matrices $U$ and $V$. Thus,
$$
\|G-\|G\|u_1\otimes v_1\|\leq \|G\|^{-1}.
$$

Our approach to verify \eqref{prop} is based on the following lemma.
\begin{lemma}\label{lemma}
Let $\ve_0$ be a sufficiently small positive constant such that $1-2\ve_0\ln(3/2)+4\varepsilon_0^2 2^{2\varepsilon_0}\ln^2(2)<1-\ve_0/2$. 
Let $G$ be a random matrix taking values in $\text{GL}(2,\bbR)$ with $|\det(G)|=1$ a.s. Suppose that there exist constants $\alpha, C>0$ such that for every $\delta>0$ we have
\begin{equation}\label{Rot}
\sup_{\|x\|=1}\bbP(|<u_1,x>|\leq\delta)\leq C\delta^\alpha.
\end{equation}
Let $2<A<B$ be two positive constants such that 
$$
A^\alpha\geq\frac{C2^{\alpha}B^{2\ve_0}}{\ve_0}.
$$
Suppose also that
$$
A\leq\|G\|\leq B
$$ 
almost surely. Then  
$$
\sup_{\|x\|=1}\bbE[\|Gx\|^{-2\varepsilon_0}]\leq 1-\varepsilon_0/4.
$$
\end{lemma}
\begin{remark}
 The matrix $U$ is rotation matrix and $u_1$ is the rotation vector, i.e. its components are $\cos(\beta)$ and $\sin(\beta)$ where $\beta$ is the random angle of rotation. Then condition \eqref{Rot} means that the law of $\beta$ is $\alpha$ regular, namely the probability that it falls within a small range of angles of size $\delta$ is of order $O(\delta^\alpha)$. 
 
 Another way to view condition \eqref{Rot} id to view $u_1$ as a random variable on the unit disc. Then \eqref{Rot} means that  the measure of an arch of length $\delta$ is $O(\delta^\alpha)$. 
 
 In both points of view the condition can be viewed as a quantitative statement that  the distribution of $u_1$ (or $\beta$)  has upper Hausdorff dimensional smaller than $\alpha$.
\end{remark}
\begin{corollary}
Let $g_j$ and $h_j$ be two sequence of independent invertible matrices in $\text{GL}(2,\bbR)$ such that $|\det(h_j)|=1$, a.s.  
Let us consider a random SVD of each $h_j$ 
$$
h_j=U_j\begin{pmatrix}
\|h_j\|& 0 \\
0 & \|h_j\|^{-1}
\end{pmatrix}	
V_j. 
$$
Let $u_j$ be the first column of $U_j$ and suppose that there are constants $\alpha,C>0$ such that for all $j$ and $\delta>0$,
$$
\sup_{\|x\|=1}\bbP(|<u_j,x>|\leq\delta)\leq C\delta^\alpha.
$$
Let $A$ and $B$ like in Lemma \ref{lemma} and assume that for all $j$,
$$
A\leq\|h_j\|\leq B
$$ 
almost surely. Then the random product $S_n=g_n\cdots g_1$ satisfies all the results in Section \ref{Sec2} if \eqref{N g} holds and we can couple $(g_j)$ and $(h_j)$ such that $\sup_j\|g_j-h_j\|_{L^1}$ is small enough. 
\end{corollary}
\begin{proof}
By Lemma \ref{lemma} and Proposition \ref{Prp} we see that Assumption \ref{Ass1} is in force for the sequence $h_j$. Now we apply the perturbation results from the previous section.
\end{proof}

\begin{proof}[Proof of Lemma \ref{lemma}]
Fix some unit vector $x$. Let us take $\delta=2/A$. Denote $J=\{|<u_1,x>|\leq\delta\}$. Then 
$$
\bbE[\|Gx\|^{-2\varepsilon_0}]=\bbE[\|Gx\|^{-2\varepsilon_0}\bbI_{J^c}]+\bbE[\|Gx\|^{-2\varepsilon_0}\bbI_J]:=I_1+I_2.
$$
To bound $I_1$ we note that 
$$
\bbE[\|Gx\|^{-2\varepsilon_0}\bbI_{J^c}]\leq (A\delta-A^{-1})^{-2\varepsilon_0}=e^{-2\varepsilon_0\ln(A\delta-A^{-1})}
=(2-A^{-1})^{-2\varepsilon_0}=e^{-2\varepsilon_0\ln(2-A^{-1})}
$$
Using the inequality $e^u\leq 1+u+u^2e^{|u|}$ we see that 
$$
e^{-2\varepsilon_0(2a-A^{-1})}\leq 1-2\varepsilon_0\ln(3/2)+4\varepsilon_0^2 2^{2\varepsilon_0}\ln^2(2)\leq 1-\ve_0/2.
$$
On the other hand, since $\|Gx\|\geq \|G\|^{-1}$ (as $\|G\|^{-1}$ is the minimal singular value) we see that 
$$
I_2\leq C\delta^\alpha B^{2\varepsilon_0}.
$$
Therefore, recalling that $\delta=2/A$,
$$
\bbE[\|Gx\|^{-2\varepsilon_0}]\leq 1-\ve_0/2+C(2/A)^\alpha B^{2\varepsilon_0}\leq 1-\varepsilon_0/4.
$$
\end{proof}

In Lemma \ref{lemma} we focused on the case of random matrices with uniformly bounded norms. The following more technical result generalizes Lemma \ref{lemma} to the unbounded case.

\begin{lemma}\label{lemma1}
Let $\ve_0$ be a sufficiently small positive constant such that $1-2\ve_0\ln(3/2)+4\varepsilon_0^2 2^{2\varepsilon_0}\ln^2(2)<1-\ve_0/2$. 
Let $G$ be a random matrix taking values in $\text{GL}(2,\bbR)$ such that $|\det(G)|=1$, a.s. Suppose that there exist  constants $\alpha,C>0$ such that for every $\delta>0$ we have
$$
\sup_{\|x\|=1}\bbP(|<u_1,x>|\leq\delta)\leq C\delta^\alpha.
$$ 
Let $A>2$ and $D\geq 0$ be constants such that 
$$
\bbP(\|G\|<A)\leq D A^{-\al}.
$$
Let $q,p$ be two conjugate exponents such that $q<\infty$ and  $A<B$ such that 
$$
\left\|\|G\|\right\|_{L^{2\varepsilon_0 p}}\leq B
$$
and
\begin{equation}\label{AB}
A^{\alpha/q}>\frac{4B^{2\ve_0}(D^{1/q}+2^{\al/q}C^{1/q})}{\ve_0}.    
\end{equation}
 Then  
$$
\sup_{\|x\|=1}\bbE[\|gx\|^{-2\varepsilon_0}]\leq 1-\varepsilon_0/4.
$$
\end{lemma}
When taking $D=0$ and $q=1$ we recover Lemma \ref{lemma} as a particular case, but we decided to state Lemma \ref{lemma} separately for the sake of clarity.

\begin{corollary}
Let $g_j$ and $h_j$ be two sequence of independent invertible matrices in $\text{GL}(2,\bbR)$ and suppose $|\text{det}(h_j)|=1$ a.s.  
 Suppose that  all $h_j$ satisfy the assumptions on $G$ in the previous lemma with the same constants $A,B,C,D,\alpha$. 
Then the random product $S_n=g_n\cdots g_1$ satisfies all the results in Section \ref{Sec2} if \eqref{N g} holds and we can couple $(g_j)$ and $(h_j)$ such that $\sup_j\|g_j-h_j\|_{L^8}$ is small enough and $\sup_j\|N(h_j)\|_{L^8}<\infty$. 
\end{corollary}

\begin{proof}[\label[Proof of Lemma \ref{lemma1}]
    Fix some unit vector $x$. Let us take $\delta=2/A$. Denote $J=\{|<u_1,x>|\leq \delta\}$ and $K=\{\|G\|<A\}$
$$
\bbE[\|Gx\|^{-2\varepsilon_0}]=\bbE[\|gx\|^{-2\varepsilon_0}\bbI_{J^c}\bbI_{K^c}]+\bbE[\|Gx\|^{-2\varepsilon_0}(\bbI_J+\bbI_{K})]=I_1+I_2.
$$
To bound $I_1$ we note that 
$$
\bbE[\|Gx\|^{-2\varepsilon_0}\bbI_{J^c}\bbI_{K^c}]\leq (A\delta-A^{-1})^{-2\varepsilon_0}=e^{-2\varepsilon_0\ln(A\delta-A^{-1})}
=(2-A^{-1})^{-2\varepsilon_0}=e^{-2\varepsilon_0\ln(2-A^{-1})}.
$$
Using again the inequality $e^u\leq 1+u+u^2e^{|u|}$ we see that 
$$
e^{-2\varepsilon_0\ln(2-A^{-1})}\leq 1-2\varepsilon_0\ln(3/2)+4\varepsilon^2 2^{2\varepsilon_0}\ln^2(2)\leq 1-\ve_0/2.
$$
On the other hand, since $\|Gx\|\geq \|G\|^{-1}$ and using also the H\"older inequality we get that
$$
I_2\leq \left\|\|G\|\right\|_{L^{2\varepsilon_0 p}}^{2\varepsilon_0}\left((\bbP(\|G\|<A)^{1/q})+(\bbP(|<u,x>|\leq\delta))^{1/q}\right).
$$
Therefore, recalling that $\delta=2/A$ and using the assumptions on the above probabilities and \eqref{AB} we conclude that
$$
\bbE[\|Gx\|^{-2\varepsilon_0}]\leq 1-\varepsilon_0/4.
$$
\end{proof}

\begin{remark}
For random invertible matrices $h_j$ in $\text{GL}(2,\bbR)$ we can just replace $h_j$ with $\tilde h_j=g_j/\sqrt{|\det(h_j)|^{1/2}}$ and then verify the conditions of Lemma \ref{lemma1} with each matrix $\tilde h_j$. Indeed such a normalization only changes $S_n$ by a constant $c_n$ that depends only on $n$. Of course, we can also perturb $h_j$.    
\end{remark}

\subsection{Application to contracting in norm  matrices}\label{Decay}
Since $d(Ax,Ay)=\|Ax\wedge Ay\|\leq \|A\|^2\|x\wedge y\|$ for every matrix $A$ we see that Assumption \ref{Ass1} holds with $n_0=1$ if 
$$
\sup_j\left(\bbE[\ln\|g_j\|]+\bbE[\ln\|g_j^{-1}\|]\right)<1/2.
$$ 
We note that such assumptions can force the norm of $g_n\cdots g_1$ to bounded above, which is in contrast to the classical theory of Furstenberg  that guarantees that the norms grow exponentially fast. However, when taking the logarithm the absolute value of the norm can still be large and so the CLT still makes sense.

\section{Extension to Markov dependent random matrices}\label{Markov}
Let us assume that $(g_j)$ is a Markov chain. In what follows we could also consider the case when $g_j=h_j(X_j)$ for some Markov chain $X_j$ and a $\text{GL}(d,\bbR)$-valued function $h_j$, but in order to avoid heavy notation we will discuss only the case when $(g_j)$ is a Markov chain.
\begin{assumption}\label{Ass2}
There exist $n_0$ and $\delta>0$ such that for every $j$ and two distinct directions $\bar x$ and $\bar y$,
$$
\bbE\left[\ln\left(\frac{d(g_{j,n_0}\bar x,g_{j,n_0}\bar y}{d(\bar x,\bar y)}\right)\Big|g_{j-1}\right]\leq -\delta
$$
almost surely,
where $g_{j,n}=g_{j+n-1}\cdots g_j$.
\end{assumption}
Note that Assumption \ref{Ass2} is stronger than Assumption \ref{Ass1}.
Under this assumption the proof of Proposition \ref{Exp Prop} proceeds similarly with the following modifications. We define $$
F_{j+m}(\bar x,\bar y)=\bbE\left[\ln\left(\frac{d(g_{j+m}\bar x, g_{j+m}\bar y)}{d(\bar x,\bar y)}\right)\Big|g_{j+m-1}\right].
$$
Then
$$
F_{j+m}(g_{j,m}\bar x,g_{j,m}\bar y)=
\bbE\left[\ln\left(\frac{d(g_{j,m+1}\bar x, g_{j,m+1}\bar y)}{d(g_{j,n}\bar x,g_{j,m}\bar y)}\right)\Big|g_j,....,g_{j+m-1}\right].
$$
Thus, defining $M_j$ like in the proof of Proposition \ref{Exp Prop} we see that it is still a martingale with uniformly bounded differences. The rest of the proof proceeds similarly.

Building on the validity of Proposition \ref{Exp Prop} the following version of Proposition \ref{Approx Prop} follows.

\begin{proposition}\label{Approx Prop1}
Suppose that for all $k$ and all $m$ the law of $(g_{k+m},...,g_{k+1})$ given $g_k$ is absolutely continuous with respect to the unconditioned laws of $(g_{k+m},...,g_{k+1})$ with uniformly bounded densities.

Then, under \eqref{N g} and Assumption \ref{Ass2}
there exists $\rho\in(0,1)$ such that for all finite $p\geq 1$,
$$
\sup_k\sup_r\rho^{-r/p}\|X_k-\bbE[X_k|g_k,g_{k-1},...,g_{k-r}]\|_{L^p}<\infty
$$
where we extend $(g_j)$ to a two sided sequence by setting $g_k=A, k<1$ for some fixed invertible matrix $A$. 
\end{proposition}

\begin{proof}
 First, since $X_k$ are uniformly bounded it is enough to prove the proposition in the case $p=1$.
Arguing like in the proof of Proposition \ref{Approx Prop}, we see that there are constants $C>0$ and $\delta\in(0,1)$ such that for all $j$ and $k>j$,
$$
\|\sigma(g_{k}, g_{j,k-1-j}x)-\sigma(g_{k}, g_{j,k-1-j}y)\|_{L^1}\leq C\delta^{k-j}.
$$
Therefore, if we denote by $\nu_{j}$ the law of $(g_1,...,g_{j-1})$ then
$$
\|f_{k}(g_k,g_{k-1},...,g_1)-\sigma(g_{k}, g_{j,k-1-j}x_0)\|_{L^1}
$$
$$
=\int\bbE\left[|f_{k}(g_k,g_{k-1},...,g_j,h_{j-1},...,h_1)-\sigma(g_{k}, g_{j,k-1-j}x_0)|g_{j-1}=h_{j-1}\right]d\nu_j(h_1,...,h_{j-1})\leq  C'\delta^{k-j}.
$$
for some constant $C'>0$.
Hence, by the minimization property of condition expectations,
$$
\|f_{k}(g_k,g_{k-1},...,g_1)-\bbE[f_{k}(g_k,g_{k-1},...,g_1)|g_k,...,g_{j}]\|_{L^1}\leq C\delta^{k-j}
$$
and the proof of the proposition is complete.   
\end{proof}
The first condition of the proposition is quite mild. For countable state Markov chains it holds when $\bbP(g_k=a,
g_{k-1}=b)\leq C\bbP(g_k=a)$ for all $a$ and $b$. For Markov chains with transition densities (with respect to their laws) it holds when the densities are uniformly bounded.

Next, to get limit theorems we need to impose appropriate mixing conditions on the chain and apply the results in \cite{NewBE} in the circumstances of Assumption \cite[Assumption 2.6]{NewBE}.

Now, we will show that Assumption \ref{Ass2} is also close under perturbations in an appropriate sense (which is relatively strong compared with the independent case).
The proof of the following result proceeds almost identically to the proof of Corollary \ref{Corr}.

\begin{proposition}
Let  $h=(h_j)$ be a Markov dependent sequence of random invertible matrices such that \eqref{N g} and Assumption \ref{Ass1} hold. Let $g=(g_j)$ be a Markov dependent sequence of invertible matrices. 

Let us assume that we can couple $h$  and $g$ such that 
$$
\te_1:=\sup_j\left\|\bbE[\tilde c(h_{j,n_0},g_{j,n_0})|g_{j-1}]\right\|_{L^\infty}<\infty.
$$
Set
$$
\te_2=\sup_j\left\|\|\nu_{j,n_0}-\nu_{j,n_0,g_{j-1}}\|_{TV}\right\|_{L^\infty}<\infty
$$
where $\nu_{j,n_0}$ is the law of $h_{j,n_0}$ and $\nu_{j,n_0,g_{j-1}}$ is the conditioned law of $h_{j,n_0}$ given $g_{j-1}$.

Then there exists a constant $\varepsilon_0=\varepsilon_0(\delta,n_0)$ such that  Assumption \ref{Ass1} holds for the sequence $(g_j)$ with some $\delta_1>0$ instead of $\delta$ and the same $n_0$ if $\max(\te_1,\te_2)<\varepsilon_0$.
\end{proposition}
One example is the case when the support of $h_j$ has sufficiently small diameter and $g_j$ takes values in a small neighborhood of the support of $h_j$. In this case we can just take the product measure of $g$ and $h$, making these processes independent so $\te_2=0$. In that case $\te_1$ is small if for all $j$ the law of $(g_{j+n_0},...,g_{j+1})$ given $g_j$ is absolutely continuous with respect to the unconditioned laws of $(g_{j+n_0},...,g_{j+1})$ with uniformly bounded densities and $\bbE[\tilde c(h_{j,n_0},g_{j,n_0})]$ is small.

\end{document}